\documentclass[12pt,leqno]{article}

\usepackage{palatino,mathrsfs}
\usepackage[mathcal]{euler}

\usepackage{xcolor}

\usepackage{graphicx} 
\usepackage{epstopdf,epsfig}

\usepackage{comment} 
\usepackage{amsmath,amsthm,amsfonts,amssymb,latexsym,amscd,enumerate,url,hyperref}
\usepackage{amssymb}
\usepackage{url}
\usepackage{graphicx}
\usepackage[all,knot]{xy}
\xyoption{arc}

\usepackage{chngcntr}
 \counterwithin*{equation}{section}
  \counterwithin{equation}{subsection}

\usepackage{multirow}

\theoremstyle{plain}
\newtheorem*{theorem}{Theorem}
\newtheorem*{lemma}{Lemma}

\theoremstyle{definition}
\newtheorem*{defn}{Definition}
\newtheorem*{example}{Example}

\theoremstyle{remark}
\newtheorem*{remark}{Remark}

\newcommand{\Gr}{\operatorname{Gr}}
\newcommand{\ord}{\operatorname{ord}}

\begin{document}

\title{Algebraic Families of Harish-Chandra Pairs}
 
\author{Joseph Bernstein\thanks{School of Mathematical Sciences, Tel Aviv University, Tel Aviv 69978,
Israel.} , \ Nigel Higson\thanks{Department of Mathematics, Penn State University, University Park, PA 16802, USA.} \ and Eyal Subag\footnotemark[2]}

 
 
 
\date{}

\maketitle


\begin{abstract}
\noindent Mathematical physicists have studied degenerations of Lie groups and their representations, which they call contractions.   In this paper we study these contractions, and also other families, within the framework of  algebraic families of Harish-Chandra modules.  We   construct a family that incorporates both a real reductive group and its compact form,   separate parts of which have been studied individually as contractions.    We    give a complete classification of generically irreducible families of Harish-Chandra modules in the case of the family associated to $SL(2,\mathbb R)$.
\end{abstract}

\setcounter{tocdepth}{2}

\section{Introduction}
The purpose of this paper is to examine from an algebraic perspective  families of Lie groups and families of representations, including examples that have long been studied in the mathematical physics literature under the name of \emph{contractions} of groups and representations.  We shall set up a general framework, but mainly we shall study the simplest examples and identify   phenomena that might deserve further study.  
 
Contractions were first studied by Segal in \cite{Segal51}, and by In\"{o}n\"{u} and Wigner in \cite{Inonu-Wigner53}. Their main goal was to study how the  symmetries of a physical system can  change  in various limiting circumstances (for example in the limit as the speed of light becomes infinite).    
This led them towards  a deformation theory for Lie groups and their representations, which they developed in various examples of physical interest. 

A simple  example is  the contraction of $SL(2,\mathbb{R})$ to the semidirect product group  $SO(2)\ltimes \mathbb{R}^2$.  
It consists of a smooth family of groups $\{G_{t}\}$ with  $G_{t} = SL(2,\mathbb{R})$ for all $t \ne 0$ and    $G_0 = SO(2)\ltimes \mathbb{R}^2$, the Cartan motion group of $SL(2,\mathbb{R})$.  It is known that any  infinite-dimensional unitary irreducible   representation of $G_0$  can be obtained as a   limit of a suitable smooth family of representations of the groups $G_t$; see   \cite{Dooley85, Subag12}. 

There is also a contraction    of $SU(2)$ to the semidirect product group $ SO(2)\ltimes \mathbb{R}^2$.  Once again,  every infinite-dimensional  unitary irreducible representation of the semidirect product   can be obtained as a suitable limit of     representations of $SU(2)$  \cite{Dooley85,Dooley83,Subag12}. But in this case the approximating representations are only defined for a \emph{discrete} sequence of parameter values converging to zero, and the representations   that approximate the given infinite-dimensional representation of $G_0$ are themselves   \emph{finite-dimensional}.   This makes it a challenge to accurately express the approximation  in mathematical terms.

We shall study these phenomena  by changing the context in various   ways.

First  we shall study families of groups that \emph{change type}  in the sense that their isomorphism classes will vary from fiber to fiber. One such family  will  be
\[
   G_t\cong \begin{cases}
SL(2,\mathbb{R}) & t>0\\
SO(2)\ltimes \mathbb{R}^2& t=0\\
SU(2)& t<0 ,
\end{cases} 
\]
which combines the two contraction families given above.

  Secondly, we shall consistently  study   families   that are param\-etrized \emph{continuously}, rather than by a discrete space. We shall  mostly deal with families over a  line, or over the  completion to a projective line.

Thirdly, we shall  study families of   representations    \emph{algebraically}, as  families of Harish-Chandra modules.  In this way we obtain sufficiently many families for our purposes, whereas there are too few continuous families of global representations in the examples we study.  Our focus will therefore be on {algebraic} families of Harish-Chandra pairs, rather than  smooth families of groups. We shall use the language of algebraic geometry and work  over the complex field; as usual, real families will be recovered using involutions.

In summary our starting point will be an algebraic family of Harish-Chandra pairs parametrized by   a line  or more generally  by a complex algebraic variety. Precise definitions will be given in Section~2, where we aim to present a formalism suitable for the study of a wide range of contractions and other families of groups.

We shall give two general constructions of algebraic families of Harish-Chandra pairs.  The first constructs  from an arbitrary  Harish-Chandra pair    a canonical \emph{deformation family}  of  pairs over the projective line using the deformation to the normal cone construction in geometry.  See Paragraph~\ref{deformation_family}. 
The second construction starts from a Harish-Chandra pair $(\mathfrak{g},K)$ that  is   \emph{symmetric}  in the sense that it is equipped with a $K$-equivariant involution of $\mathfrak{g}$ whose fixed subalgebra is the Lie algebra of $K$.  We obtain from this an algebraic \emph{contraction family} of Harish-Chandra pairs over the projective line, which is a version of the smooth family of groups displayed above.  See Paragraph~\ref{contraction_family}. Other interesting constructions are also possible.  

In both  of the algebraic families  that we shall study the generic member of the family will be (isomorphic to) a fixed Harish-Chandra pair $(\mathfrak{g},K)$.   So when we study families of Harish-Chandra modules we shall be  in particular studying families of Harish-Chandra modules for the fixed pair  $(\mathfrak{g},K)$.    It is of course common in representation theory to encounter representations in families rather than individually. Usually this happens within the context of parabolic induction, and the parameter space is an affine variety.  One aspect of our study is an investigation of how these parabolically induced affine families may be compactified over a projective variety (in the cases studied here, this is just the projective line).  

  We  shall introduce  real structures on the complex families that we consider (associated to a  real structure on the Harish-Chandra pair $(\mathfrak{g},K)$ that we start from) in the usual way.   In a companion paper  we shall use    Jantzen filtration techniques to recover from our complex algebraic families of Harish-Chandra modules the discrete families of finite-dimensional representations of compact groups originally considered in the mathematical physics literature.  In this way  the formalism of algebraic families and Harish-Chandra pairs places the discrete approximation phenomenon with\-in an algebraic context, and construction of Jantzen recovers the phen\-om\-enon from the algebra.

The same techniques have other uses.  For instance we can  recover the ``Mackey bijection'' \cite{Mackey75,Higson08,Higson11}  between the (tempered or admissible) dual of $SL(2,\mathbb{R})$ and that of  its Cartan motion group $SO(2) \ltimes \mathbb{R}^2$.   Our approach is potentially quite general.

 Our focus is on families of Harish-Chandra pairs, but it is interesting to construct  families of \emph{groups} that give rise to families of Harish-Chandra pairs.  The deformation to the normal cone construction does this for the first of our families.  The situation for the  contraction family associated to a symmetric pair  is more complicated, but for many classical groups  we shall  give in Section 3 the construction of an underlying family of groups.  In the case of groups defined by bilinear forms, it is natural to use the (projectivization of the) space of all forms as a parameter space, and then perhaps restrict to interesting curves within this space, and this is what we do. We do not know how to handle the general case.

In Section 4, which is independent of Section 3, we shall give a classification theorem for  generically irreducible families of Harish-Chandra modules for the contraction family over $\mathbb{C}P^1$ that is associated to the pair 
\[
(\mathfrak{g}, K) = \left (\mathfrak{sl}(2,\mathbb{C}), SO(2,\mathbb{C})\right ).
\]
Apart from the expected parameters involving the Casimir element and $K$-types, we shall explain how simple topological invariants fit into the classification of families over $\mathbb C P^1$.  (We shall also see that there are  too many generically irreducible families to admit a reasonable classification, and we shall impose  some restrictions on the families considered, again of a topological nature.)

The topics covered here are obviously just a sampling from among the phenomena involving families of representations that one might study, and even for families related to $SL(2,\mathbb{R})$  there is still plenty to be done.  We hope that this paper will serve as an invitation to the further study of families from an algebraic point of view.

The first and third authors were partially supported by ERC grant 291612. Part of the work on this project was done at Max-Planck Institute for Mathematics, Bonn. We would like to thank MPIM for the very stimulating atmosphere. The second author was partially supported by NSF grant DMS-1101382.

\section{Algebraic Families}
\label{section2}
 
In this section we shall give  the family versions of a number of standard  constructions in representation theory. Namely, we shall define  families of Lie algebras, families of algebraic groups, families of Harish-Chandra pairs, and  families of  Harish-Chandra  modules. At the end of the section we shall consider involutions and real structures.

All of the  objects and morphisms appearing in this section and the rest of the paper  will be algebraic and defined over the field of complex numbers.  For example, by  a vector bundle we shall mean a complex algebraic vector bundle, by a section we shall mean an algebraic section and so on. 

By a \emph{variety} we shall always mean an irreducible,  nonsingular,  quasi-projective,  complex algebraic variety, and for any   variety, $X$, we shall denote by $O_X$ the structure sheaf of regular functions on $X$.

\subsection{Families of Lie algebras }
\begin{defn}
Let $X$ be a variety. \textit{An algebraic family of   Lie algebras   over $X$}  is  a locally free sheaf of $O_X$-modules, that is equipped with $O_X$-linear Lie brackets which make it a sheaf of Lie algebras.  
A \emph{morphism} of algebraic families of Lie algebras over $X$ is a morphism of $O_X$-modules  that commutes with the Lie brackets.
\label{def1}
\end{defn}

All of the families of Lie algebras that we shall consider will have  finite-dimensional fibers. In the finite-dimensional case, an algebraic family of complex Lie algebras   is the same thing as the sheaf of sections of  an algebraic vector bundle whose fibers are equipped with Lie algebra structures that   vary algebraically.

\subsubsection{Constant and nonconstant families}
 
Let $X$ be a  variety and let $\mathfrak{g}$ be  a  complex Lie algebra.  The \emph{constant family over $X$ with fiber $\mathfrak{g}$} is $O_X\otimes \mathfrak{g}$ (tensor product over $\mathbb{C}$). 
A \emph{locally constant}  algebraic family of complex Lie algebras is a family that is locally  isomorphic to a constant family. 

Here is a simple example of an algebraic family that is \emph{not} locally constant.
Let  $X = \mathbb{C}$ and let $\mathfrak g$ be any  non-abelian complex Lie algebra. Make the sheaf $O_X \otimes \mathfrak{g}$  into an     algebraic family of  Lie algebras over the variety $\mathbb{C}$ by means of the formula
\[
 [s_1,s_2](z) = z [s_1(z),s_2(z)]
\]
for the Lie bracket of sections.  The fiber over $0$ is $\mathfrak{g}$ with the trivial Lie bracket.
 
\subsubsection{The deformation family associated to  a Lie  subalgebra} 
\label{deformation_family}

 Let $\mathfrak{g}$ be a complex Lie algebra and let $\mathfrak{k}$ be a Lie subalgebra of $\mathfrak{g}$.   We shall construct a family over the variety 
\[
X = \mathbb{C} 
\]
 with fibers
 \[
 \begin{cases}
\mathfrak{g} &z \neq 0\\
\mathfrak{k}\ltimes  \mathfrak{g}/\mathfrak{k}& z =0  
\end{cases}
\]
(it is essentially the  deformation to the normal cone construction in geometry, as described in   \cite{Fulton84,Fulton84-2}, for example).   Form the sheaf of regular functions on $X$ with values in $\mathfrak{g}$, and then form the subsheaf   that consists of functions whose value at $ 0\in X$  belongs to $\mathfrak{k}$.  One can show that this subsheaf is, in its own right,   locally free.  It is an algebraic family of Lie algebras in our sense, which we shall call the \emph{deformation family} associated to the inclusion $\mathfrak{k}\subseteq \mathfrak{g}$.

The restriction of the deformation family to the complement of $0\in X$ is the  constant family with fiber $\mathfrak{g}$.  The fiber at $0\in X$ is the semidirect product Lie algebra
$
\mathfrak{k}\ltimes  \mathfrak{g}/\mathfrak{k}$, as indicated above.

\subsubsection{The contraction family associated to a symmetric Lie subalgebra}
\label{contraction_family}

  The following construction of a second family over $X=\mathbb{C}$  is based upon the  In\"{o}n\"{u}-Wigner contractions of Lie algebras studied in \cite{Inonu-Wigner53}.  Let $\mathfrak{g}$ be a  complex Lie algebra, let 
\[
\theta\colon \mathfrak{g} \longrightarrow \mathfrak{g}
\]
be an involution, and let  
\begin{equation*}
\label{eq-cartan-decomp}
\mathfrak{g} = \mathfrak {k} \oplus \mathfrak{p}
\end{equation*}
be the associated  ``Cartan'' decomposition  into $+1$ and $-1$ eigenspaces of $\theta$, respectively.   Let $X = \mathbb{C}$.  Give the sheaf $O_X \otimes \mathfrak{g}$ the structure of a nonconstant algebraic family of Lie algebras by using the decomposition of sheaves
\[
O_X \otimes \mathfrak{g} = O_X \!\otimes\! \mathfrak{k} \, \oplus \, O_X\! \otimes \!\mathfrak{p}
\]
and by defining the Lie bracket operation on sections belonging to the individual summands by 
\[
[\eta, \zeta](z) =  \begin{cases}
z [\eta(z), \zeta(z)]_{\mathfrak g}   & \text{if $\zeta$  and $\eta$ are sections of $ O_X\! \otimes \!\mathfrak{p}$}   \\
\phantom{i} [\eta(z), \zeta(z)]_{\mathfrak g}  & \text{otherwise}.
\end{cases}
\]

\subsubsection{Base change}
Given an   algebraic family  $ { \boldsymbol{\mathfrak{h}}}$ of complex Lie algebras over $Y$ and a morphism of  varieties 
\[
\psi:X\longrightarrow Y,
\]
we can pull back  $\boldsymbol{\mathfrak{h}}$ to the  sheaf  of $O_X$-modules
\[
\psi^* \boldsymbol{\mathfrak{h}} = O_X\otimes_{\psi^{-1}O_Y}\psi^{-1}\boldsymbol{\mathfrak{h}} .
\]
  It has the structure of an algebraic family of Lie algebras over $X$.

\begin{example} 
Let $\mathfrak {g}$ be a Lie algebra equipped with an involution and a Cartan decomposition 
$\mathfrak{g} = \mathfrak{k}\oplus\mathfrak{p}$ 
 as in Paragraph \ref{contraction_family}.  Since  
$
\mathfrak{k} \ltimes \mathfrak{p} \cong \mathfrak{k}\ltimes  \mathfrak{g}/\mathfrak{k},
$
the  contraction family of Lie algebras  strongly resembles the deformation  family from  Paragraph \ref{deformation_family}.  But it is not isomorphic to it since    the deformation   family  is isomorphic to the sheaf $O_{\mathbb{C}} \otimes \mathfrak{g}$ over $\mathbb{C}$ with Lie bracket 
\[
[\eta, \zeta](z) =  \begin{cases}
z^2 [\eta(z), \zeta(z)]  & \text{if $\zeta$  and $\eta$ are sections of $ O_X\! \otimes \!\mathfrak{p}$}   \\
\phantom{z} [\eta(z), \zeta(z)]  & \text{otherwise}  
\end{cases}
\]
on homogeneous sections (note the appearance of $z^2$ in place of the monomial $z$ that is used in \ref{contraction_family}).
Denote this latter family over $\mathbb{C}$ by   $\boldsymbol{\mathfrak{g}}$, and   let   $\boldsymbol{\mathfrak{h}}$ be the contraction family over $\mathbb{C}$ from  Paragraph \ref{contraction_family}.  As sheaves of $O_X$-modules, $\boldsymbol{\mathfrak{g}}$  and  $\boldsymbol{\mathfrak{h}}$ are the same. However the identification is not a morphism of algebraic families of Lie algebras. Instead,   $\boldsymbol{\mathfrak{g}}$   is isomorphic to        $\psi^*\boldsymbol{\mathfrak{h}}$, where
\[
\psi \colon \mathbb{C} \longrightarrow \mathbb{C} , \quad \psi(z) = z^2.
\]
\end{example}

\subsubsection{Families of  modules}

Let $\boldsymbol{\mathfrak{g}}$ be an algebraic family of complex Lie algebras over $X$.  We want to distinguish between the concepts $\boldsymbol{\mathfrak{g}}$-\emph{module}   and    \emph{family of representations}  of $\boldsymbol{\mathfrak{g}}$.   A module, loosely speaking, is just a collection of modules over the individual fibers of $\boldsymbol{\mathfrak{g}}$:

\begin{defn}
Let $\boldsymbol{\mathfrak{g}}$ be an algebraic family of complex Lie algebras  over $X$. A \emph{$\boldsymbol{\mathfrak{g}}$-module} is  a quasicoherent $O_X$-module, $\mathcal{F}$, together with  a morphism of $O_X$-modules $\rho:\boldsymbol{\mathfrak{g}}\otimes _{O_X}\mathcal{F}\longrightarrow \mathcal{F}$   that  respects the Lie brackets.
\end{defn}

In contrast, for families of representations we require a degree of continuity from fiber to fiber: 

\begin{defn}
Let $\boldsymbol{\mathfrak{g}}$ be an algebraic family of complex Lie algebras  over $X$. A \emph{family of representations} over $\boldsymbol{\mathfrak{g}}$  is  a $\boldsymbol{\mathfrak{g}}$-module    that is flat as an $O_X$-module.
\end{defn}

Flatness is the correct technical notion, but it is not easy to understand  from a geometric perspective.  However in the examples that arise most immediately in representation theory  one is presented with  $\boldsymbol{\mathfrak{g}}$-modules that decompose into countable direct sums of coherent $O_X$-modules, and it is worth recalling that a coherent and flat $O_X$-module is locally free; see for example \cite[Proposition 9.2 (e)]{Hartshorne1977}.

\begin{example} If $\mathfrak g$ is a complex semisimple Lie algebra, and if $\mathfrak b =\mathfrak h \oplus \mathfrak n$ is a Borel subalgebra, then as $\lambda$ ranges over $\mathfrak h ^*$  the Verma modules 
 \[
 V_\lambda = \mathcal{U}(\mathfrak g ) \otimes _{\mathcal{U}(\mathfrak b )} \mathbb C_\lambda
 \]
  combine to form a module for the constant family of Lie algebras over $\mathfrak h^*$ with fiber $\mathfrak g$.  Other examples, including examples of modules over nonconstant families, will be given later. 
\end{example}

\subsection{Families of algebraic groups}
 \label{groupscheme}

\begin{defn}
 An \emph{algebraic family of  groups} over a  variety $X$ is a smooth morphism of varieties 
 \[
 \boldsymbol{G} \longrightarrow  X
 \]
  that carries the structure of a group scheme over $X$ (for details about group schemes see \cite[Tag 022R]{stacks-project}).   
    A morphism of algebraic families of groups over $X$ is a morphism of varieties over $X$ that is compatible with group structures. \end{defn}

\begin{remark}
In our context, thanks to our simplifying assumptions on the varieties we are studying,  the smoothness condition on the morphism $\pi$ from $ \boldsymbol{G} $ to $  X$ can be checked pointwise: it is equivalent to the surjectivity of the differential of $\pi$ on every tangent space. That is, a morphism $\pi$ is smooth if and only if it is a submersion in the $C^\infty$-sense.  See for example \cite[Proposition 10.4]{Hartshorne1977}.
\end{remark}

\begin{example}
If $G$ is a     complex   algebraic group and $X$ is a   variety, then of course we can form the constant family $\boldsymbol{G} = G\times X$.    
 \end{example}

\begin{example}
Let $G$ be a complex   algebraic  group, and let $K$ be an algebraic subgroup.  Associated to the inclusion of $K$ into $G$ there is the \emph{deformation to the normal cone}, as in \cite[Chapter 5]{Fulton84-2} for example, which is an algebraic family of groups over $\mathbb C \mathbb P ^1$.   This is the group-theoretic counterpart of the Lie algebraic deformation family in Paragraph \ref{deformation_family} (we defined the deformation family of complex Lie algebras over $\mathbb{C}$, but exactly the same definition may be given over $\mathbb{CP}^1$).
\end{example}

\begin{remark} 
 The construction of a group-theoretic counterpart of the contraction family in Paragraph \ref{contraction_family} is a more delicate matter, which will be considered in Section~\ref{section3}.
 \end{remark}

\begin{defn}
Let $\boldsymbol{G}$ be an algebraic family of  algebraic groups over $Y$, and let $\psi \colon X \to Y$ be a morphism of varieties.  The base change of  $\boldsymbol{G}$ with respect to $\psi$ is the  group scheme  
\[
\psi^*\boldsymbol{G}=X\times_{Y}\boldsymbol{G}
\]
over $X$.\end{defn}

\subsubsection{The associated family of Lie algebras}

Given an algebraic family of   groups   $\boldsymbol{G}$ over $X$ as in Section~\ref{groupscheme},  the associated algebraic family of Lie algebras,  denoted by  $\operatorname{Lie}(\boldsymbol{G})$, is the sheaf of vertical, right-invariant vector fields on $\boldsymbol G$, or in other words the sheaf of  $O_X$-linear  derivations $\xi$   of $O_{\boldsymbol G}$ for which the diagram 
\[
\xymatrix{
O_{\boldsymbol G}\ar[d]_{\xi}   \ar[r]  & 
	O_{\boldsymbol G} \otimes _{O_X} O_{\boldsymbol G} \ar[d]^{\xi\otimes 1}
	\\
O_{\boldsymbol G}    \ar[r]  & 
	O_{\boldsymbol G} \otimes _{O_X} O_{\boldsymbol G}  }
\]
is commutative (the horizontal maps are   multiplication   in $\boldsymbol{G}$).

According to our assumptions the identity section $X \to \boldsymbol{G}$ embeds $X$ as a smooth subvariety of the smooth variety $\boldsymbol{G}$. The sheaf $\operatorname{Lie}(\boldsymbol{G})$ is isomorphic to the sheaf of sections of the normal bundle of $X$ in $\boldsymbol{G}$  by restriction of vector fields to  $X$.

\subsubsection{Action of a family of groups on an $O_X$-module}

 If  $\boldsymbol{G}$ is an algebraic family of   groups   over $X$, and if $\mathcal{F}$ is a sheaf of $O_X$-modules, then an \emph{action} of $\boldsymbol{G}$ on $\mathcal{F}$ is a morphism of sheaves of $O_X$-modules
 \[
\xymatrix{
 \mathcal{F} \ar[r] & O_{\boldsymbol{G}} \otimes _{O_X} \mathcal{F}
 }
 \] 
 that is compatible with the multiplication and unit operations on $\boldsymbol{G}$ in the usual way.  This is a special case of the concept of equivariant sheaf; for further information  see  \cite{Bernstein-Lunts,Vistoli05}.

\begin{defn}
A \emph{family of representations} of $\boldsymbol{G}$ is a flat, quasicoherent sheaf of $O_X$-modules that is equipped with an action of $\boldsymbol{G}$.
\end{defn}

\subsubsection{Action of the associated family of Lie algebras}

If $\boldsymbol{G}$ acts on $\mathcal{F}$, then there is an induced action of $\operatorname{Lie}(\boldsymbol{G})$ defined by 
\[
\xymatrix{
\operatorname{Lie}(\boldsymbol{G})\otimes _{O_X} \mathcal{F} 
\ar[r] & 
\operatorname{Lie}(\boldsymbol{G})\otimes _{O_X} O_G \otimes_{O_X} \mathcal{F}  \ar[r]  & 
O_G\otimes _G \mathcal {F} \ar[r]  &
\mathcal{F}
}
\]
where the first arrow comes from the action of $\boldsymbol{G}$, the second from the action of $\operatorname{Lie}(\boldsymbol{G})$ by derivations, and the third from restriction along the identity section $X \to \boldsymbol{G}$.
 
\subsubsection{The adjoint action}

The identity section
\[
X \longrightarrow \boldsymbol{G}
\]
is invariant under the 
adjoint action 
\[
\boldsymbol{G} \times _X \boldsymbol{G} \longrightarrow \boldsymbol{G} ,
\]
 and so there is an induced adjoint action of $\boldsymbol{G}$ on the normal bundle for the image of the identity section, or in other words on the associated family of Lie algebras:
 \[
 \operatorname{Lie} (\boldsymbol{G}) \longrightarrow O_{\boldsymbol{G}} \otimes _{O_X}  \operatorname{Lie} (\boldsymbol{G}) 
 .
 \]
The action of $ \operatorname{Lie} (\boldsymbol{G})$ associated, in turn, to this action of $\boldsymbol{G}$ is the LIe bracket operation on $\operatorname{Lie}(\boldsymbol{G})$.

\subsubsection{Reductive families and admissible modules} 
\label{sec-reductive-families-and-admissibility}

\begin{defn}
 Let ${\boldsymbol{K}}$ be an algebraic family of groups over $X$. 
 We shall say that ${\boldsymbol{K}}$  is \emph{reductive}  if  for every $x\in  X$ the fiber  ${\boldsymbol{K}}|_x$ is a complex reductive algebraic group.  
\end{defn}

So a reductive family over $X$  is the same thing as a smooth, reductive group scheme over $X$. 

 \begin{defn}
Let ${\boldsymbol{K}}$ be a reductive  family of groups over $X$ and let $\mathcal{F}$ be a family of representations of $\boldsymbol{K}$.   We shall say that $\mathcal{F}$ is   \emph{admissible}   if    the sheaf of $O_X$-modules
$[ \mathcal{L}\otimes_{O_X} \mathcal{F}]^{\boldsymbol{K}}$ is locally free and of finite rank, for any      family $\mathcal{L}$ of representations of $\boldsymbol{K}$ that is locally free and of finite rank  as an $O_X$-module. 
\end{defn}

\begin{example} 
If  ${\boldsymbol{K}}=K\times X$  is a constant reductive family, then   $\mathcal{F}$ is admissible  if and only if the sheaf 
  \[
  \mathcal{H}om_{\boldsymbol{K}}(O_X\otimes_{\mathbb{C}} V\,,\, \mathcal{F}) \cong   [ \, V^*  \otimes _{\mathbb C} \mathcal F\, ]^{K}
  \]
   is coherent for every finite-dimensional representation $V$ of $K$.  In this case there is a canonical isotypical decomposition 
   \[
   \bigoplus _{\tau\in \widehat K}   \mathcal{F}_\tau
   \stackrel\cong\longrightarrow \mathcal F ,
   \]
indexed by the equivalence clases of irreducible (algebraic) representations of $K$, where
     \[
\mathcal{F}_\tau =  
V_\tau \otimes _{\mathbb C} \mathcal{H}om_{\boldsymbol{K}}(O_X\otimes_{\mathbb{C}} V_\tau\,,\, \mathcal{F}) .
   \]
   \end{example}

\subsection{Families of Harish-Chandra pairs}
 \label{para-hc-pair-family}
 
Finally we come to the main definition in this section: that of an algebraic  family of Harish-Chandra pairs.  There are small variations in the literature on the familiar concept of an individual Harish-Chandra pair.  The following definition will make clear which one we wish to adopt here (take $X$ to be a single point).

\begin{defn}
Let $X$ be a   variety. An \emph{algebraic family of  Harish-Chandra pairs}  over  $X$ consists of the following data:
\begin{enumerate}[\rm (a)]

\item  an algebraic family  $\boldsymbol{\mathfrak{g}}$ of   Lie algebras over $X$, and

\item an algebraic family ${\boldsymbol{K}}$ of   groups over $X$, 
\end{enumerate}
together with an  action   of $\boldsymbol{K}$ on  $\boldsymbol{\mathfrak{g}}$ by automorphisms, and 
a $\boldsymbol{K}$-equivariant  embedding of algebraic families of   Lie algebras over $X$, 
\[
j:\operatorname{Lie}(\boldsymbol{K})\longrightarrow {\mathfrak{g}},
\]
such that  the action of $\operatorname{Lie}(\boldsymbol{K})$ on $\boldsymbol{\mathfrak{g}}$ via $\operatorname{ad}_{\boldsymbol{\mathfrak{g}}} \circ j$ coincides with the differential of the action of $\boldsymbol{K}$ on $\boldsymbol{\mathfrak{g}}$.
\label{FHCP}
\end{defn}

 There are the obvious notions of morphism of families Harish-Chandra pairs over $X$, base change and morphism of families of Harish-Chandra pairs over different bases.  We omit the details.

\subsubsection{Deformation and contraction families of Harish-Chandra pairs}
\label{CanonicalFamily}
In Paragraph \ref{deformation_family} we associated to any pair of Lie algebras $\mathfrak{k}\subseteq  \mathfrak{g}$   a  deformation family $\boldsymbol{\mathfrak{g}}$ of Lie algebras over $\mathbb C $.   If $\mathfrak{k}$ is the Lie algebra of a complex algebraic group  $K$,  and if  $(\mathfrak{g},K)$ is a Harish-Chandra pair,  then we obtain from the deformation family an algebraic family of Harish-Chandra pairs 
$
(\boldsymbol{\mathfrak{g}},\boldsymbol{K})
$
in which $\boldsymbol{K}$ is the constant algebraic family of groups over $\mathbb C$ with fiber $K$.

If in addition $\mathfrak k \subseteq \mathfrak{g}$ is the fixed-point subalgebra of a $K$-equivariant involution on $\mathfrak{g}$,  and if instead of the deformation family we consider the contraction family of Lie algebras from Paragraph~\ref{contraction_family}, then we obtain a second family of Harish-Chandra pairs $(\boldsymbol{\mathfrak{g}}, \boldsymbol{K})$  over $\mathbb C $, in which $\boldsymbol K$ is once again the constant family of groups with fiber $K$.

 \subsection{Algebraic families of Harish-Chandra modules}

 \begin{defn}
Let $({\mathfrak{g}},{\boldsymbol{K}})$ be an algebraic family of   Harish-Chandra pairs over $X$. An \emph{algebraic family of Harish-Chandra modules} for $({\mathfrak{g}},{\boldsymbol{K}})$ is a flat, quasicoherent  $O_X$-module, $\mathcal{F}$, that is equipped with
\begin{enumerate}[\rm (a)]

\item an  action of $\boldsymbol{K}$ on $\mathcal{F}$, and

\item an  action  of  $\boldsymbol{\mathfrak{g}}$ on $\mathcal{F}$,

\end{enumerate}
such that  the action morphism
\[
 \boldsymbol{\mathfrak{g}}\otimes_{O_X}\mathcal{F}\longrightarrow \mathcal{F}
\]
 is $\boldsymbol{K}$-equivariant, and such that the differential of the $\boldsymbol{K}$-action in (a) is equal to the composition of the inclusion of $\operatorname{Lie}(\boldsymbol{K})$ into $\boldsymbol{\mathfrak{g}}$   with the action of   $\boldsymbol{\mathfrak{g}}$ on $\mathcal{F}$. 

\label{FHCM}
\end{defn}

\subsubsection{Quasi-admissible families of Harish-Chandra modules} 

Let $({\mathfrak{g}},{\boldsymbol{K}})$ be  an  algebraic family of   Harish Chandra pairs over $X$, and assume that $\boldsymbol{K}$ is a reductive algebraic family of groups over $X$.

\begin{defn}
An algebraic family   $\mathcal{F}$ of Harish-Chandra modules is  \emph{quasi-admissible} if the $\boldsymbol{K}$-action on $\mathcal F$ is admissible.
\end{defn}

\begin{remark}
\label{rem-admissible-versus-gen-admissible}
One can define a family to be \emph{admissible} if it is    quasi-admissible and   finitely generated.  But a more useful concept is probably that of a  \emph{generically admissible} family, which may be defined in analogy with the concept of \emph{generically irreducible} family; see Paragraph~\ref{par-generically-irreducible} below.
\end{remark}

\subsection{Real structures on algebraic families} 

In this section  we shall recall some definitions related to real structures on varieties, and make the natural definitions of real structures on algebraic families of Lie algebras and groups.

\subsubsection{Real structures on varieties}

Recall that if $X$ is a variety, then its \emph{complex conjugate}  $\overline X$ is the variety whose underlying  topological space is the same as that of $X$, and whose structure sheaf $O_{\overline{X}}$ is the complex conjugate of the sheaf $O_X$; this is the same sheaf of rings as $O_X$, but equipped with the complex conjugate scalar multiplication.   

The operation of complex conjugation is a functor from the category of varieties to itself, and the  composition of complex conjugation  with itself is   the identity functor.

\begin{defn} 
Let $X$ and $Y$ be varieties.  An \emph{antiholomorphic morphism} from $X$ to $Y$ is a morphism of algebraic varieties  
 $X\longrightarrow \overline{Y}$.
 \end{defn}
 
 \begin{defn}
 A \emph{real structure}, or \emph{antiholomorphic involution}, on $X$ is a morphism  $\sigma_X:X\longrightarrow \overline{X}$ such that  the composition 
\[
\xymatrix{
X\ar[r]^{\sigma_X} & \overline{X}\ar[r]^{\overline{\sigma_X}} &  X} 
\]
is the identity. \end{defn}

Compare for example \cite[Chapter 1]{Borel1991} or \cite[Chapter 11]{Springer} for all this.
 
\subsubsection{Real structures on $O_X$-modules}

Let $X$ be a variety and let $\mathcal{F}$ be a sheaf of  $O_X$-modules.  The complex conjugate sheaf $\overline{\mathcal{F}}$ is a sheaf of $O_{\overline{X}}$-modules, and complex conjugation is a functor from $O_X$-modules to $O_{\overline{X}}$-modules.

\begin{defn}
Given a real structure $\sigma_X$ on $X$, a  \emph{real structure}, or  \emph{antiholomorphic involution},  on $\mathcal{F}$ is a morphism  of $O_X$-modules
\[
\sigma_\mathcal{F}:\mathcal{F}\longrightarrow \sigma_X^*\overline{\mathcal{F}}
\]
 such that  the composition 
\[
\xymatrix@C=30pt{
\mathcal{F}   \ar[r]^-{\sigma_{\mathcal{F}}} &
\sigma_X^*\overline{\mathcal{F}}   \ar[r]^-{\sigma_X^*[ \overline{\sigma_{\mathcal{F}}}]} &
\sigma_X^*\Bigl [\overline{\sigma_X^*\overline{\mathcal{F}}}\Bigr ] \ar[r]^-\cong & \mathcal{F}} 
\]
is the identity morphism. 
\end{defn}
 
\begin{example}
If $\mathcal{F}$ is the sheaf of sections of a vector bundle $V$, then a real structure on $\mathcal F $ is the same thing as a real structure on the variety underlying $V$  that yields a morphism of vector bundles
\[
\xymatrix{
{V}\ar[d]_{\pi}\ar[r]^{\sigma_{{V}}}  &\overline{{V}}\ar[d]^{\overline{\pi}} \\
  X	\ar[r]_{\sigma_X}&\overline{X} .}
    \]
 \end{example}

\subsubsection{Real structures on algebraic families of Lie algebras and groups}

Let $X$ be a variety and let us specialize the above discussion to our algebraic familes. 

If $\boldsymbol{G}$ is an algebraic family  of groups over   $X$, then $\overline{\boldsymbol{G}}$ is a family of complex algebraic groups over the conjugate variety $\overline{X}$.

Similarly, if $\boldsymbol{\mathfrak{g}}$  is an algebraic family of complex Lie algebras over $X$ then the complex conjugate sheaf $\overline{\boldsymbol{\mathfrak{g}}}$  is an algebraic family of complex Lie algebras over $\overline{X}$.

Finally, if $({\boldsymbol{\mathfrak{g}}},{{\boldsymbol{K}}})$ is an algebraic family of Harish-Chandra pairs over $X$ then $({\overline{\boldsymbol{\mathfrak{g}}}},{\overline{\boldsymbol{K}}})$  is an algebraic family of Harish-Chandra pairs over $\overline{X}$.

 \begin{defn}
 We shall refer to the families above as the \emph{conjugate families} associated to $\boldsymbol G$, $\boldsymbol{\mathfrak g}$ and $(\boldsymbol{\mathfrak g}, \boldsymbol{K})$. A \emph{real structure} in each case is a morphism $\sigma$ for which the composition of $\sigma$ with $\overline{\sigma}$ is the identity.
\end{defn}

\subsubsection{Real families associated to real structures}

The \emph{set of real points} of a real structure on $X$ is the set of fixed points for $\sigma_X$ in the underlying topological space of $X$ (it may of course be empty).  

Given a real structure on an algebraic family of groups or Lie algebras over $X$, we obtain by restriction a family of complex groups or Lie algebras over the real points in $X$.  Each member of the family carries its own real structure, and by further passing to fixed sets we obtain families of Lie groups or real Lie algebras over the real points in $X$.  It is in this way that we shall recover families originally studied in the mathematical physics literature.

\section{Algebraic families of  classical groups}\label{section3}

 In this section we shall construct algebraic families of symmetric pairs of groups whose underlying families of Harish-Chandra pairs are the contraction families considered in Paragraphs~\ref{contraction_family} and \ref{CanonicalFamily}.  The construction will not apply to every case, but for instance it will apply to the examples
 \begin{enumerate}[\rm (a)]
 \item $\bigl(\,GL(q{+}p,\mathbb{C})\,,\,GL(q,\mathbb{C}){\times} GL(p,\mathbb{C})\,\bigr )$
 \item $\bigl(\,O(q{+}p,\mathbb{C})\,,\,O(q,\mathbb{C}){\times} O(p,\mathbb{C})\,\bigr )$
 \item $\bigl ( \, Sp(2n,\mathbb{C})\,,\,GL(n,\mathbb{C})\, \bigr )$
\end{enumerate}
as well as the determinant-one versions of (a) and (b) (we do not know how to handle the most general case). 

Thus we shall begin with a complex affine algebraic group $G$, together with an involution $\theta$ with fixed subgroup $K$, and construct (in the cases above) an algebraic family $\boldsymbol{G}$  of groups over $\mathbb{CP}^1$.   The fiber over all but two points in $\mathbb{CP}^1$ will be isomorphic to $G$, and the remaining two fibers will be isomorphic to the semidirect product $K\ltimes \mathfrak{p}$, where $\mathfrak{p}$ is the minus one eigenspace of $\theta$ as it acts on the Lie algebra of $G$ (we treat $\mathfrak p$ as an additive group, using its vector space structure).  The family $\boldsymbol{G}$ will carry an involution, and  the fixed-point subfamily   will be the  constant family  of groups with fiber $K$.

Our families will carry natural real structures that are compatible with the standard real structure on $\mathbb{CP}^1$.  We obtain a family of real groups over $\mathbb{RP}^1\subseteq \mathbb{CP}^1$.  For example in case (i) the family  has the form
\[
\boldsymbol{G}^{\sigma}|_{x}\cong \begin {cases} U(p,q) &x >0\\
  U(p){\times}U(q) \ltimes \mathfrak{p}^\sigma& x=0,\infty \\
 U(p{+}q) & x <0 ,
 \end{cases}
 \]
 where $\mathfrak{p}^\sigma$ is the minus one eigenspace of the Cartan involution of the real Lie algebra $\mathfrak{u}(p,q)$ with fixed subalgebra $\mathfrak{u}(p){\times}\mathfrak{u}(q)$, and $x$ is the usual coordinate on $\mathbb{RP}^1$.

When $p=q=1$ the determinant-one version of this family will give 
\[
\boldsymbol{G}^{\sigma}|_{x}\simeq\begin {cases} SU(1,1) & x>0\\
 U(1)\ltimes \mathfrak{p}^\sigma& x=0,\infty \\
 SU(2) & x <0 ,
 \end{cases}
 \]
which is the family that we mentioned in the introduction. 
 
 \subsection{The general construction}\label{sec33}
 
 We shall start with an algebraic group $G$ and an involution $\theta$ with fixed subgroup $K$.  We shall assume that $G$ is embedded in a larger complex affine algebraic  group $H$.  We shall obtain an algebraic family of groups by  first  conjugating $G$ inside of $H$ by inner automorphisms of $H$, so as to obtain a family of subgroups of $H$ that are all isomorphic to $G$, and then forming a closure so as to obtain the family that we want.

We are interested in families of groups that contain the \emph{constant} family $K$ as a subfamily, and so we shall consider only inner automorphisms that fix $K$ (in fact it will be sufficient for our purposes to consider only automorphisms that fix $K$ pointwise, and indeed only some of these automorphisms).

To form the closure, we shall work in the Grassmannian $\Gr_d (\mathfrak{h})$ of linear subspaces of $\mathfrak{h}$ of dimension $d = \dim(\mathfrak{g})$ where $\mathfrak{h}$ is the Lie algebra of $H$.  

Let  $Z\subseteq H$ be  a subvariety consisting of elements in $H$  that centralize $K$ (we emphasize that $Z$ need not be the full centralizer of $K$ in $H$, and it need not even be a subgroup).  Denote by 
\[
X_0 \subseteq \Gr_d (\mathfrak{h})
\]
the image inside the Grassmannian of the morphism 
\[
\alpha\colon Z \longrightarrow \Gr_d (\mathfrak{h})  
\]
defined by $\alpha (z) = \operatorname{Ad}_z[\mathfrak{g}]$, and denote by 
\[
X\subseteq \Gr_d (\mathfrak{h}) 
\]
 the  closure of $X_0 $ in $ \Gr_d (\mathfrak{h})$.  Consider now the restriction to $X$ of the tautological bundle $E$ over $\Gr_d(\mathfrak{h})$, as in the diagram:
\[
\xymatrix{
  \boldsymbol{\mathfrak{g}}\,\ar[d]  \ar@{^{(}->}[r] & E \, \ar@{^{(}->}[r] \ar[d] & \Gr_d(\mathfrak{h}){\times} \mathfrak{h}\ar[d]\\
  X\, \ar@{^{(}->}[r] & \Gr_d(\mathfrak{h})\, \ar@{=}[r]  &\Gr_d(\mathfrak{h})   .
  } 
  \]
The fibers of $\boldsymbol{\mathfrak{g}} $ are Lie subalgebras of $ \mathfrak {h}$, and we obtain in this way an algebraic family $\boldsymbol{\mathfrak {g}}$ of Lie algebras over the projective variety $X$.

As for algebraic families of  groups, if  $G$ is connected (or if $K$ meets every component of $G$), then  the subgroup $\operatorname{Ad}_z[G]\subseteq H$ depends only on the image of $z$ under the morphism $\alpha$, and so we obtain a collection of algebraic subgroups of $H$ parametrized by $X_0$.  It may be checked on a case by case basis for our examples that this is an algebraic family of groups, namely a subfamily of the constant family with fiber $H$, and that the closure of this family in $\Gr_{d}(\mathfrak{h})\times H$ is an algebraic family of groups $\boldsymbol{G}$ over $X$, whose family of Lie algebras is the family $\boldsymbol{\mathfrak{g}}$ above.  We shall summarize   the calculations in one case in the next section. 

Our construction of real structures will be as follows.  We shall begin with   a real structure $\sigma$ on $G$ that commutes with the involution $\theta$ of $G$.\footnote{It is worth recalling that  real structures on a reductive $G$, which  are the main groups of interest to us, are in correspondence, up to conjugation, with involutions on $G$, in such a way that the involution and real structure commute, and the composition determines a compact form of $G$; see \cite[Theorems 3\&4, pp.230--231]{OnishchikVinberg}.}    We shall assume that $\sigma$ extends to $H$ and also determines a real structure on $Z\subseteq H$.   Then all of the families described above acquire real structures from $\sigma$.

\subsection{Special cases}
 \label{sec-special-cases-of-group-families}
 Let us examine the case of the pair 
 \[
 G = GL(q{+}p,\mathbb{C})\quad \text{and} \quad K=GL(q,\mathbb{C}){\times} GL(p,\mathbb{C}).
 \]
   We shall embed $G$ diagonally in the product $H = G {\times} G$, and  we shall take  $Z\subseteq H$ to be the set of all pairs $z= (g,g^{-1})$, where  
   \[
   g = \begin{pmatrix} \mu I_q & 0 \\ 0 & \nu I_p\end{pmatrix}   ,
   \]
  and where $\mu,\nu\in \mathbb{C}^{\times}$.    The Lie subalgebra
$
  \operatorname{Ad}_{z}[ \mathfrak{g}] \subseteq   \mathfrak{h}
$
decomposes as 
\[
  \operatorname{Ad}_{z}[ \mathfrak{g}]  = \mathfrak{k} \oplus \mathfrak{p}_z,
  \]
   where $\mathfrak{p}_z$ consists of all pairs of matrices of the form
   \[
 \Bigl (
 \begin{bmatrix} 
 0 & \mu^2B \\
 \nu^{2} C & 0
 \end{bmatrix} 
 \, , \, 
  \begin{bmatrix} 
 0 &\nu^{2}  B \\
\mu^2 C & 0
 \end{bmatrix}   
 \Bigr ) .
\]
This  only depends on $\mu^2\nu^{-2} \in \mathbb{C}^{\times}$, and the variety $X_0\subseteq \Gr_d(\mathfrak{h})$ from the previous section is isomorphic to $\mathbb{C}^\times$ in this way.  

The closure of $X_0$ is isomorphic to $\mathbb{CP}^1$; the points in $\mathbb{CP}^1$ with homogeneous coordinates $[0,1]$  and $[1,0]$ correspond to the spaces  $\mathfrak{p}_0$ and $\mathfrak{p}_\infty$ of matrix pairs of the types
\[
 \Bigl (
 \begin{bmatrix} 
 0 & 0 \\
 C & 0
 \end{bmatrix} 
 \, , \, 
  \begin{bmatrix} 
 0 & B \\
   0 & 0
 \end{bmatrix}   
 \Bigr )
 \quad
 \text{and}
 \quad 
  \Bigl (
 \begin{bmatrix} 
 0 &  B \\
0 & 0
 \end{bmatrix} 
 \, , \, 
  \begin{bmatrix} 
 0 & 0 \\
  C & 0
 \end{bmatrix}   
 \Bigr )
 ,
 \]
 respectively.   We obtain an algebraic family of Lie algebras over $X= \mathbb{CP}^1$ that is isomorphic  over $\mathbb C\subseteq \mathbb{CP}^1$ to the contraction family from Paragraph~\ref{contraction_family} associated to the involution 
 \[
 \theta = \operatorname{Ad}_{\left(\begin{smallmatrix} 
I & 0 \\ 0 & -I
 \end{smallmatrix} \right )}
\colon \mathfrak{g} \longrightarrow \mathfrak{g} .
\]

\begin{remark}
In Paragraph~\ref{contraction_family} we defined the contraction family over the affine line rather than over the projective line.  The computation above shows that the family can be extended to the projective line in at least the special case considered there,  but it is a simple matter to do so in general.
 First, the formula
\[
[\eta, \zeta](w) =  \begin{cases}
\frac 1w [\eta(w), \zeta(w)]  & \text{if $\zeta$  and $\eta$ are sections of $ O_X\! \otimes \!\mathfrak{p}$}   \\
\phantom{i} [\eta(w), \zeta(w)]  & \text{otherwise}.
\end{cases}
\]
defines an algebraic family over the complement of $0$ in $\mathbb{CP} ^1$.   Second, the formula 
\[
\eta  \longmapsto 
\begin{cases}
z \cdot \eta  & \text{if   $\eta $ is a section of $ O_X\! \otimes \!\mathfrak{p}$}   \\
\phantom{i} \eta  & \text{if   $\eta $ is a section of $ O_X\! \otimes \!\mathfrak{k}$}
\end{cases}
\]
defines an isomorphism from the family over $\mathbb{CP} ^1 \setminus \{ \infty\}$ to the family over   $\mathbb{CP} ^1 \setminus \{ 0\}$,  when both are restricted to $\mathbb{CP} ^1 \setminus \{ 0,\infty\}$.  We can use this isomorphism to glue the two families (sheaves) together.  \end{remark}

Continuing with the special case, the closure of the family of groups $G_z$ over $X_0$ is the algebraic family of groups over $X= \mathbb{CP}^1$, a subfamily of the constant  family with  fiber $H$, whose fibers over $0$ and $\infty$ are the  subgroups 
\[
G_{0,\infty} = \{\, k + X : k \in K \,\,\text{and}\,\, X \in \mathfrak{p}_{0,\infty}\,\}
\]
 (recall that $K$ is diagonally embedded in $H=G{\times}G$).  As algebraic groups, these are isomorphic to the semidirect products 
$
 K \ltimes  \mathfrak{p}_{0,\infty}$.

 The antiholomorphic involution  of $G$ defined by 
\[
\sigma (g^*) = \left [ \begin{smallmatrix} I & 0 \\ 0 & -I \end{smallmatrix} \right ] g^{-1} \left [ \begin{smallmatrix} I & 0 \\ 0 & -I \end{smallmatrix} \right ] 
\]
(where $g^*$ is the conjugate-transpose) commutes with $\theta$ and the corresponding group of real points is $U(p,q)$. If we extend $\sigma$ to   $H=G{\times} G$ by the formula 
\[
\sigma \colon (g_1,g_2) \longmapsto (\sigma(g_2), \sigma(g_1))
\]
then $\sigma$ maps $Z\subseteq H$  to its conjugate, determines the standard real structure on  $X=\mathbb{CP}^1$, for which the antiholomorphic involution is complex conjugation on homogeneous coordinates,  and gives real structures on the groups $\boldsymbol{G}\vert_x$, with $x\in \mathbb {RP}^1$, exactly as described at the beginning of this section.

\begin{remark}
All of the constructions pass to the determinant-one subgroups of $G$ and $H$, and in the particular case $p=q=1$ we obtain the family of real groups
\[
\boldsymbol{G}^{\sigma}|_{x}\cong\begin {cases} SU(1,1) & x>0\\
 U(1)\ltimes \mathfrak{p}^\sigma& x=0,\infty \\
 SU(2) & x <0 ,
 \end{cases}
 \]
  over $\mathbb{RP}^1$ we were seeking.
\end{remark}

\section{A classification problem}\label{section4}

In this final section of the paper  we shall study a classification problem in order to explore a little further the concept of algebraic family of Harish-Chandra modules in a simple, concrete case. 

The setup is as follows. We shall study one particular algebraic family of Harish-Chandra pairs, namely the family $(\boldsymbol{\mathfrak{g}},{\boldsymbol{K}})$ from Section~\ref{sec-special-cases-of-group-families}, with $p=q=1$. In this case the fibers of $\boldsymbol{\mathfrak{g}}$ are generically isomorphic to $\mathfrak{sl}(2,\mathbb{C})$ and  $\boldsymbol{K}$ is the constant family of groups with fiber $\mathbb{C}^{\times}$ (embedded as the diagonal subgroup in $SL(2,\mathbb C)$). The sheaf of Lie algebras $\boldsymbol{\mathfrak{g}}$ decomposes as a direct sum of invertible sheaves 
\[
\boldsymbol{\mathfrak{g}}= \boldsymbol{\mathfrak{g}}_{2}\oplus \boldsymbol{\mathfrak{g}}_{0}\oplus \boldsymbol{\mathfrak{g}}_{-2}
\]
according to the action of $\boldsymbol{K}$, and we shall analyze algebraic families of Harish-Chandra modules by breaking down the $\boldsymbol{\mathfrak{g}}$-action into three parts according to this decomposition. The family of Lie algebras associated to $\boldsymbol{K}$ coincides with  $\boldsymbol{\mathfrak{g}}_0$ and it is isomorphic to the constant family  $O_X\otimes_{\mathbb{C}} \operatorname{Lie}(\mathbb{C}^{\times})$. 

An algebraic family of Harish-Chandra modules $\mathcal{F}$ has a $\boldsymbol{K}$-isotypical decomposition
  \[
 \mathcal{F} = \bigoplus_{n\in \mathbb Z}  \mathcal{F}_n .
\]
 The action of $\boldsymbol{K}$ on $\mathcal{F}$ is of course given by the isotypical decomposition, and  since $\boldsymbol{\mathfrak{g}}_0 = \operatorname{Lie}(\boldsymbol{K})$, the action of 
 $\boldsymbol{\mathfrak{g}}_0$ is determined too.
The full action of $\boldsymbol{\mathfrak{g}}$  determines, and is determined by, morphisms of sheaves
\[
\boldsymbol{\mathfrak{g}}_{\pm 2}\otimes _{O_X} \mathcal{F}_n \longrightarrow  \mathcal{F}_{n\pm 2}.
\]
They satisfy     compatibility conditions that are related in part to the fact that $\boldsymbol{\mathfrak{g}}_0$ is the Lie algebra of $\boldsymbol{K}$.  For the generically irreducible families $\mathcal F$ that we shall study,  each non\-zero  $\mathcal{F}_n$ is in fact an invertible sheaf, and the degrees of these sheaves further constrain the morphisms above.

  The main lessons learned from our investigation will be as follows:

\begin{enumerate}[\rm (a)]

\item In the context of families it is   appropriate to classify \emph{generically irreducible} families of Harish-Chandra modules, defined below, rather than  irreducible families.

  \item  There are invariants for the classification problem, namely $K$-types, and infinitesimal characters (in our case the infinitesimal characters  will be associated to certain elements of the ring $\mathbb{C}[z,z^{-1}]$), that are similar to those in the standard classification problem for individual Harish-Chandra modules.
  
  \item There are also  additional  geometric invariants, namely the degrees of the invertible sheaves $\mathcal F_n$.
  
\item  These invariants together are  enough to solve  some cases of the  classification problem, but they are not enough in general.  In fact the general problem is unwieldy without the imposition of further hypotheses beyond generic irreducibility. \end{enumerate}
The classification results we obtain will be used in the sequel to this paper when we consider representations of real groups and the contraction families of representations from mathematical physics.

\subsection{Generically irreducible and quasi-simple families}\label{sec401}
\interfootnotelinepenalty=10000
\subsubsection{Generically irreducible modules}
\label{par-generically-irreducible}
Let $(\boldsymbol{\mathfrak{g}},{\boldsymbol{K}})$ be an algebraic family of Harish-Chandra pairs    over a variety $X$.  Informally speaking,    an algebraic family of Harish-Chandra modules $\mathcal{F}$ for $(\boldsymbol{\mathfrak{g}},{\boldsymbol{K}})$   is \emph{generically irreducible} if for almost any $x\in X$ the fiber $\mathcal{F}|_x$ is an irreducible  $(\boldsymbol{\mathfrak{g}}|_x,\boldsymbol{{K}}|_x)$-module.\footnote{Formally, an algebraic family of Harish-Chandra modules $\mathcal{F}$ over $(\boldsymbol{\mathfrak{g}},{\boldsymbol{K}})$   is \emph{generically irreducible}  if    $\Bbbk \otimes_{O_{X}}  \mathcal F $ is   irreducible as a module  for the Lie algebra  $\Bbbk \otimes_{O_{X}}{\boldsymbol{\mathfrak{g}}}$, where $\Bbbk$ is the algebraic closure of the   field of rational functions on $X$.}   For the family of Harish-Chandra pairs that we will work with in this section, generic irreducibility means that  all except at most countably many fibers are irreducible.

\subsubsection{Quasi-simple families}
Let $\boldsymbol{\mathfrak{g}}$ be an algebraic family of Lie algebras  over $X$.  The  \emph{sheaf of universal enveloping algebras}   $\mathcal{U}(\boldsymbol{\mathfrak{g}})$ is  the sheaf of $O_X$-algebras characterized by    the usual universal property: it is equipped with a morphism 
\[
\boldsymbol{\mathfrak{g}} \longrightarrow \mathcal{U}(\boldsymbol{\mathfrak{g}})
\]
that is compatible with Lie  and commutator brackets, and  is initial among such morphisms.   

In the context of algebraic families, the   {sheaf of universal enveloping algebras} is locally free as a sheaf of $O_X$-modules (in fact there is a Poincar\'e--Birkhoff--Witt isomorphism from the sheaf of symmetric algebras   $S(\boldsymbol{\mathfrak{g}})$ to  the sheaf of enveloping algebras).

\begin{defn}
Let  $\mathcal{Z}(\boldsymbol{\mathfrak{g}}) $ be the center of  $\mathcal{U}(\boldsymbol{\mathfrak{g}}) $.
A $(\boldsymbol{\mathfrak{g}},{\boldsymbol{K}})$-module $\mathcal{F}$ is \emph{quasisimple}  if   the morphism
\[
\mathcal{U}(\boldsymbol{\mathfrak{g}}) \longrightarrow \operatorname{End}(\mathcal{F})
\]
maps $\mathcal{Z}(\boldsymbol{\mathfrak{g}})$ into the subsheaf $O_X\cdot I_{\mathcal{F}}\subseteq  \operatorname{End}(\mathcal{F})$. 
\end{defn}

This means that a family is quasisimple if  every fiber has an infinitesimal character. The usual Schur's Lemma argument implies:

\begin{lemma} A generically irreducible quasi-admissible $(\boldsymbol{\mathfrak{g}},{\boldsymbol{K}})$-module is quasisimple. \qed
\end{lemma}

\subsection{The contraction family   for  SL(2)}

Throughout the rest of Section~\ref{section4} we shall  focus our attention on a specific  family $(\boldsymbol{\mathfrak{g}},\boldsymbol{K})$ of Harish-Chandra pairs over $X= \mathbb{C}\mathbb{P}^1$, namely the family associated to the determinant one case of  the family of groups constructed in Section~\ref{sec-special-cases-of-group-families}, with $p=q=1$. 

To be explicit, $\boldsymbol{\mathfrak{g}}$ is the following subfamily of the constant family of Lie algebras over $\mathbb{C}\mathbb{P}^1$ with fiber   $\mathfrak{sl}(2,\mathbb{C})\times \mathfrak{sl}(2,\mathbb{C})$:
\[
 \left \{ \Bigl( \begin{bmatrix}
a & \alpha b\\
\beta c & -a
\end{bmatrix} ,
\begin{bmatrix}
a & \beta b \\
\alpha c & -a
\end{bmatrix} , [\alpha:\beta] \Bigr) 
	\in \mathfrak{sl}(2,\mathbb C){\times} \mathfrak{sl}(2,\mathbb C){\times}\mathbb{CP}^1\, \right \}.
\]
It can be checked that   
\[
\boldsymbol{\mathfrak {g}}|_{z}\cong\begin {cases} \mathfrak{sl}(2,\mathbb{C}) & z\neq 0,\infty \\
 \mathbb{C}^{\times}\ltimes \mathbb{C}^2& z=0,\infty ,
 \end{cases}
 \]
where $z=\alpha/\beta$. The family  $\boldsymbol{K}$ is the constant family of groups  with fiber
$
K =  \Delta H=\{\,{(h,h)}|h\in H\,\} 
$
 where 
 \[ 
H = \left\{   \begin{bmatrix}
a & 0\\
0 & a^{-1}
\end{bmatrix}  \in SL (2,\mathbb{C})    \right\}  .
\]
Our aim is to  classify the  equivalence classes of generically irreducible and  quasi-admissible $(\boldsymbol{\mathfrak{g}},\boldsymbol{K})$-modules in this case.  

Restricted to the affine line $\beta \ne 0$,  this is the contraction family discussed in Paragraph~\ref{contraction_family}. We could have chosen to work with the deformation family from Paragraph~\ref{deformation_family}, extended to a family over $\mathbb {CP}^1$.  The outcome would be essentially the same.

\subsection{Isotypical decomposition of $\boldsymbol{\mathfrak{g}}$}
\label{sec-isotypical-decomp-g}

In the following lemma we shall write $k_a$ for the diagonal element $(h_a,h_a)$, where 
\[
h_a =  \begin{bmatrix}
a & 0\\
0 & a^{-1}
\end{bmatrix} .
\]

 \begin{lemma}
The sheaf $\boldsymbol{\mathfrak{g}}$ decomposes into a direct sum
\[
\boldsymbol{\mathfrak{g}}= \boldsymbol{\mathfrak{g}}_{2}\oplus \boldsymbol{\mathfrak{g}}_{0}\oplus \boldsymbol{\mathfrak{g}}_{-2}
\]
under the action of $ \boldsymbol{K}$, in such  a way that
\[
\operatorname{Ad}_{ k_a}\cdot \sigma_n= a^n\cdot \sigma_n 
 \]
  for any section $\sigma_n$ of $\boldsymbol{\mathfrak{g}}_n$.  Each of the sheaves  $\boldsymbol{\mathfrak{g}}_n$ is   invertible,  with  
  \[
 \deg(\boldsymbol{\mathfrak{g}}_{ 2}) = -1, \quad  \deg(\boldsymbol{\mathfrak{g}}_0) = 0 \quad \text{and} \quad \deg(\boldsymbol{\mathfrak{g}}_{ -2}) = -1.
  \] 
\end{lemma}
\begin{proof} For the existence of the weight decomposition see Paragraph~\ref{sec-reductive-families-and-admissibility}.
The formulas 
\begin{align*}
X \colon [\alpha: \beta] &\longmapsto \Bigl( \begin{bmatrix}
0 & \alpha/\beta \\
0 & 0
\end{bmatrix} ,
\begin{bmatrix}
0 & 1 \\
0 & 0
\end{bmatrix} \Bigr)
\\
Y \colon [\alpha: \beta] & \longmapsto \Bigl( \begin{bmatrix}
0 & 0 \\
\beta/\alpha    & 0
\end{bmatrix} ,
\begin{bmatrix}
0 &  0 \\
1  & 0
\end{bmatrix} \Bigr)
\end{align*}
define nowhere vanishing  rational  sections of $\boldsymbol{\mathfrak{g}}_{2}$ and $\boldsymbol{\mathfrak{g}}_{-2}$  respectively, $X$ has a  simple poles at $\beta=0$, $Y$ has a  simple poles at $\alpha=0$, while the formula 
\[
H \colon [\alpha: \beta]   \longmapsto \Bigl( \begin{bmatrix}
1& 0\\
0& -1
\end{bmatrix} ,
\begin{bmatrix}
1 & 0\\
0 & -1
\end{bmatrix} \Bigr)
\]
defines a nonwhere vanishing regular section of  $ \boldsymbol{\mathfrak{g}}_{0}$.  These sections determine the degrees.
\end{proof}

\subsection{Invariants}
\label{sec-parameters}

We shall start by attaching  the usual invariants, more or less, to generically irreducible algebraic families of Harish-Chandra modules, namely an invariant associated to the \emph{Casimir subsheaf}  of $\mathcal{Z}(\boldsymbol{\mathfrak{g}})$, which we shall define below,   and an invariant  consisting of the set of the $K$-types appearing in the isotypical decomposition of a module.  

\subsubsection{The Casimir sheaf}
\label{sec-casimir-section}
The sheaf $\mathcal{U}(\boldsymbol{\mathfrak{g}})$ is filtered by the usual notion of order in enveloping algebras. This filtration descends to the  center $\mathcal{Z}(\boldsymbol{\mathfrak{g}})$. 

\begin{defn}
The  \emph{Casimir sheaf} $\mathcal {C}\subseteq \mathcal{Z}(\boldsymbol{\mathfrak{g}})$  is the subsheaf of the order $2$ part of 
$
\mathcal{Z}(\boldsymbol{\mathfrak{g}})  $   consisting of sections that act trivially on the  family of trivial modules.
\end{defn}

The Casimir sheaf is an invertible sheaf, isomorphic to $O(-2)$ (this will be clear from the explicit construction of rational sections in the next paragraph).  By Schur's lemma its action on a generically irreducible algebraic family of Harish-Chandra modules is determined by a morphism of $O_X$-modules 
\[
\mathcal C \longrightarrow O_X.
\]
This is our first invariant.

\subsubsection{Sections of the Casimir sheaf}
\label{sec-canonical-casimir-section}

The sections $X$, $H$ and $Y$ in the proof of the lemma in Section~\ref{sec-isotypical-decomp-g}  form an $\mathfrak{sl}(2)$ triplet
\[
[H,X] = 2 X , \quad [H,Y] = -2 Y \quad \text{and } \quad [X,Y] =  H
\]
(but since $X,Y$ and $H$ are only \emph{rational} sections, the relations do not  of course imply that $\boldsymbol{\mathfrak{g}}$ is isomorphic to a constant family).  
 The formula 
\[
C= H^2 + 2XY + 2YX = H^2 + 2H + 4YX
\]
defines a  rational section of the Casimir sheaf, which we shall call   the  \emph{ Casimir section}.  From the formula it is evident that 
\[
\operatorname{ord}_{0} (C) = -1\quad \text{and} \quad  \operatorname{ord}_{\infty} (C) = -1,
\]  while $\operatorname{ord}_{z} (C) = 0$ elsewhere. Hence:

\begin{lemma}\label{lem12}
Let $\mathcal{F}$ be a generically irreducible  quasi-admissible $(\boldsymbol{\mathfrak{g}},\boldsymbol{K})$-module. The  Casimir section, $C$, acts on $\mathcal F$  by multiplication by a rational  function on $X$ of the form
\[ 
\pushQED{\qed}
c_1z +c_0 +c_{-1}z^{-1}. 
\qedhere
\popQED
\]
\end{lemma}

%

  \subsubsection{Weights of generically irreducible modules}

 In this section we shall study the restriction to $\boldsymbol{K}$ of a generically irreducible and quasi-admissible algebraic family of Harish-Chandra modules  for $(\boldsymbol{\mathfrak{g}},\boldsymbol{K})$.

Recall that the algebraic family of groups $\boldsymbol{K}$ in our algebraic family of  Harish-Chandra pairs  $(\boldsymbol{\mathfrak{g}},\boldsymbol{K})$ is the constant family with fiber the diagonal subgroup $K=\Delta H$  in $H{\times} H$, where  $H$ is the group of diagonal matrices in $SL(2,\mathbb C)$.  The irreducible representations of $K$ are the one-dimensional weights 
\[
k_a \longmapsto  a^n 
\]
 for $n\in \mathbb Z$.   We shall use the word ``weight'' rather than ``$K$-type'' in what follows, and if $\mathcal{F}$ is an algebraic family of $(\boldsymbol{\mathfrak{g}},\boldsymbol{K})$-modules, then we shall refer to the $K$-isotypical   decomposition 
 \[
 \mathcal{F} = \bigoplus_{n\in \mathbb Z}  \mathcal{F}_n
 \]
of   $\mathcal F$, as in Paragraph~\ref{sec-reductive-families-and-admissibility},  as its  \emph{weight decomposition}.

Since a quasi-admissible family $\mathcal F$ is generically irreducible if and only if the fiber $\mathcal{F}|_x$ is an irreducible  $(\boldsymbol{\mathfrak{g}}|_x,\boldsymbol{{K}}|_x)$-module for all but countably many points $x$, there are (infinitely many) fibers $\mathcal F \vert _x$ which are  irreducible as    $(\boldsymbol{\mathfrak{g}}\vert _x ,\boldsymbol{K}\vert _x)$-modules.  These have    the same weights  as $\mathcal{F}$ itself, and hence:

\begin{lemma}\label{lem4.3.3}
 The list of weights with multiplicities, of a generically irreducible quasi-admissible $(\boldsymbol{\mathfrak{g}},\boldsymbol{K})$-module coincides  with the list of weights, with multiplicities  of some irreducible admissible $(\mathfrak{sl}(2,\mathbb{C}),H )$-module.  
 \qed
 \end{lemma}

 We list all possibilities for  weights of such modules  in Table \ref{tab1new}.\footnote{We have included limit discrete series as discrete series, since for the purposes of the current work there is really no difference between the two.}  Each weight   appears with multiplicity at most one and all the weights of a given  irreducible admissible module share the same parity.   This list of weights is our second invariant of a    generically irreducible quasi-admissible $(\boldsymbol{\mathfrak{g}},\boldsymbol{K})$-module.
  \begin{table}[ht]
 \begin{center}
\begin{tabular}{|c|c|} 
\hline
 {\bf Module}   &   {\bf Weights}  \\
\hline
even principal series  &   $2\mathbb{Z}$  \\
\hline
odd principal series&   $2\mathbb{Z}+1$   \\
\hline
positive discrete   series,   $\ell\ge 1$ &   $   \ell + 2\mathbb{N}$   \\
\hline
negative discrete   series, $\ell \le -1$    &   $\ell   -  2 \mathbb N $ \\
\hline
{}\qquad  finite-dimensional, $k\ge 0$\qquad{}  & \qquad  $k, k{-}2,\dots , -k$\qquad{} \\
\hline
 \end{tabular}
 \end{center}
 \caption{ The weights  of admissible irreducible $(\mathfrak{sl}(2,\mathbb{C}),H)$-modules.}\label{tab1new}
\end{table}

\subsection{Further invariants}

The action of the Casimir sheaf  and the weight decomposition  are not enough to completely determine a  generically irreducible algebraic family of Harish-Chandra modules $\mathcal F$ up to isomorphism.  That is, they are not enough to determine  the morphisms of sheaves
\begin{equation}
\label{eq-first-action}
\boldsymbol{\mathfrak{g}}_2 \otimes \mathcal F_n \longrightarrow \mathcal{F}_{n+2}
\end{equation}
and 
\begin{equation}
\label{eq-second-action}
\boldsymbol{\mathfrak{g}}_{-2} \otimes \mathcal F_{n+2} \longrightarrow \mathcal{F}_{n}
\end{equation}
for all $n$.   The action of the Casimir sheaf \emph{is} sufficient to determine the composition 
\[
\boldsymbol{\mathfrak{g}}_{-2} \otimes \boldsymbol{\mathfrak{g}}_2 \otimes \mathcal F_n \longrightarrow \mathcal{F}_{n},
\]
as we shall see below, and so each  of the morphisms \eqref{eq-first-action} and \eqref{eq-second-action} determines the other.  We need further invariants, and, as we shall see, in at least some cases these are conveniently provided by the degrees of the sheaves $\mathcal F_n$.

\subsubsection{Degrees of the weight components}

 If $\mathcal {F} = \oplus_{n\in \mathbb Z} \mathcal{F}_n$  is the weight decomposition of a generically irreducible quasi-admissible $(\boldsymbol{\mathfrak{g}},\boldsymbol{K})$-module, then we have seen that each nonzero $\mathcal{F}_n$ is a locally free sheaf of $O_X$-modules of rank  one.  Since $X{=}\mathbb{CP}^1$, the sheaf   of $O_X$-modules   $\mathcal{F}_n$ is therefore determined  up to  isomorphism by its degree, $\deg (\mathcal{F}_n)$.   
 
 The function $n\mapsto \deg(\mathcal{F}_n)$, which is defined on the set of weights of $\mathcal{F}$, will be    our third invariant.

\begin{lemma} 
Let $\mathcal {F} = \oplus \mathcal{F}_n$  be the weight decomposition of a generically irreducible quasi-admissible $(\boldsymbol{\mathfrak{g}},\boldsymbol{K})$-module.  If  $\mathcal{F}_{n}$ and $\mathcal{F}_{n+2}$ are nonzero, then:
\begin{enumerate}[\rm (a)]  
\item $ \deg ( \mathcal{F}_n )-1 \le   \deg ( \mathcal{F}_{n+2})   \le\deg ( \mathcal{F}_n ) + 1  .$
\item  If  $ \deg ( \mathcal{F}_n )-1 =   \deg ( \mathcal{F}_{n+2})$ ,
then the morphism
\[
\boldsymbol{\mathfrak{g}}_2 \otimes \mathcal F_n \longrightarrow \mathcal{F}_{n+2}
\]
is an isomorphism.
\item If $\deg ( \mathcal{F}_{n+2})=\deg ( \mathcal{F}_n )+1$,
 then the morphism
\[
\boldsymbol{\mathfrak{g}}_{-2} \otimes \mathcal F_{n+2} \longrightarrow \mathcal{F}_{n}
\]
is an isomorphism.
\end{enumerate}
\end{lemma}

\begin{proof}
In the generically irreducible case the action determines non-zero morphisms of sheaves
\[
\boldsymbol{\mathfrak{g}_2} \otimes \mathcal{F}_n \longrightarrow   \mathcal{F}_{n+2}
\]
and
\[
\boldsymbol{\mathfrak{g}_{-2}} \otimes \mathcal{F}_{n+2} \longrightarrow   \mathcal{F}_{n}
\]
Such nonzero morphisms exist only when 
\[
\deg(\boldsymbol{\mathfrak{g}_2}) + \deg( \mathcal{F}_n)\le   \deg( \mathcal{F}_{n+2}) 
\]
and
\[
\deg(\boldsymbol{\mathfrak{g}_{-2}}) + \deg( \mathcal{F}_{n+2})\le   \deg( \mathcal{F}_{n}) ,
\]
respectively.  So the first part of the lemma follows from the the fact that the sheaves $\boldsymbol{\mathfrak{g}}_{\pm 2}$ have degree minus one. The rest of the lemma follows from the fact that for $X{=}\mathbb{CP}^1$, any nonzero morphism between two invertible sheaves of $O_X$-modules with equal degrees is an isomorphism. 
\end{proof}

\subsection{Basis for a generically irreducible module}
\label{sec4.8}

Let $\mathcal F$ be a  generically irreducible quasi-admissible  $(\boldsymbol{\mathfrak{g}},\boldsymbol{K})$-module, with weight decomposition
\[
\mathcal{F} = \bigoplus_{n\in\mathbb Z}\mathcal{F}_n
\]
We have seen that each  nonzero weight space $\mathcal F_n$ is a locally free sheaf of $O_X$-modules of rank one. Keeping in mind that $X = \mathbb{CP}^1$, we find that  each non-zero weight space is isomorphic to some $O(k)$, and we can choose a section of the following sort  that will be convenient for computations.

\begin{lemma} 
Let $\mathcal F$ be a  generically irreducible $(\boldsymbol{\mathfrak{g}},\boldsymbol{K})$-module.  
\begin{enumerate}[\rm (i)]

\item Each nonzero weight space  $\mathcal{F}_n$  has  a rational section $f_n$  that is regular except perhaps at $z=\infty$, and also nowhere vanishing except perhaps    at $z=\infty$. Moreover   $\ord_\infty (f_n) = \deg (\mathcal F_n)$.  

\item The section $f_n$  is unique, up to multiplication by a nonzero complex scalar.  \qed
\end{enumerate}
\end{lemma}

 Now consider the   rational sections $X$, $H$ and $Y$ of the weight spaces $\boldsymbol{\mathfrak{g}}_{2}$, $\boldsymbol{\mathfrak{g}}_{0}$ and $\boldsymbol{\mathfrak{g}}_{-2}$ introduced in Paragraph~\ref{sec-casimir-section}.
If we fix sections $f_n$ as in the lemma above, then
 \begin{equation*}
\begin{aligned}
Hf_n &  =nf_{n} \\
Xf_n  & =A_n  f_{n+2} \\
Yf_{n+2} &= z^{-1}B_{n} f_{n}
\end{aligned}
\end{equation*}
for some polynomial functions  $A_n,B_n\in \mathbb{C}[z]$, where as before $z=\frac{\alpha}{\beta}$. 

The degrees of $A_n(z)$ and $B_n(z)$ are constrained by the degrees of the sheaves $\mathcal{F}_n$ and $\mathcal{F}_{n+2}$ (whenever these sheaves are nonzero).  Namely the formulas above lead to the inequalities
\[
1 -  \deg (\mathcal {F}_n) \ge   \deg (A_n ) - \deg(\mathcal {F}_{n+2}) 
\]
and 
\[
1 -  \deg (\mathcal {F}_{n+2}) \ge   \deg (B_n ) -  \deg(\mathcal {F}_{n}) .
\]

Now from  its formula we find that  the     Casimir section $C$ (see Paragraph~\ref{sec-casimir-section} again) acts on $f_n$  as multiplication by the polynomial function
\[
z \longmapsto  (n^2+2n)+ 4A_n (z) B_n (z)z^{-1} .
\]
This must of course be independent of $n$, which is a constraint on the collection of all $A_n$ and $B_n$. 
\subsection{The case of varying degrees}
 
Let $\mathcal F$ be a  generically irreducible  $(\boldsymbol{\mathfrak{g}},\boldsymbol{K})$-module. We have seen that 
\[
 \deg ( \mathcal{F}_n )-1 \le   \deg ( \mathcal{F}_{n+2})   \le\deg ( \mathcal{F}_n ) + 1   
\]
whenever $\mathcal {F}_n$ and $\mathcal F_{n+2}$ are nonzero.  The following computation shows that if $\deg(\mathcal F_n)$ always differs from $\deg(\mathcal F_{n+2})$, then the invariants that we have identified serve to parametrize generically irreducible algebraic families of Harish-Chandra modules.

\begin{theorem}
The isomorphism class of a   generically irreducible  $(\boldsymbol{\mathfrak{g}},\boldsymbol{K})$-module $\mathcal F$  for which $\deg(\mathcal F_n)$ always differs from $\deg(\mathcal F_{n+2})$ is   determined by:
\begin{enumerate}[\rm (i)]
\item  the list of weights of $\mathcal{F}$,
\item the degrees $\deg(\mathcal{F}_n)$, and 
\item the action of the Casimir section.
\end{enumerate}
\end{theorem}

\begin{proof}
We know the weights of $\mathcal F$ constitute a finite, semi-infinite or infinite arithmetic progression in $\mathbb Z$ with interval $2$.  Consider any two consecutive weight spaces $\mathcal{F}_n$ and $\mathcal {F}_{n+2}$.  If 
\[
 \deg ( \mathcal{F}_{n+2} )  =      \deg ( \mathcal{F}_{n}) -1 
 \]
 then, referring to the notation and the inequalities in the previous paragraph, we find that $A_n$ is a constant polynomial.  It must be nonzero since $\mathcal F$ is generically irreducible. As for $B_n$, the formula at the end of previous section tells us that it is determined by the action of the Casimir section and the value of the constant $A_n$.
 
  The other possibility is that 
\[
  \deg ( \mathcal{F}_{n})  =   \deg ( \mathcal{F}_{n+2} ) - 1,
 \]
in which case  $B_n$ is a nonzero constant, while $A_n$ is determined by this constant and  the value of the Casimir section.

We can adjust the sections $f_n$ by multiplication with nonzero scalars so that in either case the constants above are always $A_n=1$ in the first case and $B_n=1$ in the second case.  Given any two such generically irreducible modules with the same weights,  degrees of weight spaces,  and actions of the Casimir section, there is a unique isomorphism of sheaves of $O_X$-modules that corresponds these adjusted sections $f_n$, and it is an isomorphism of  $(\boldsymbol{\mathfrak{g}},\boldsymbol{K})$-modules. 
\end{proof}
 
\subsection{The case of equal degrees}

At the other extreme are the generically irreducible algebraic families of Harish-Chandra modules for which all the degress $\deg (\mathcal{F}_n)$ are equal.  We  shall  briefly study them in this paragraph.

In this case an  analysis like the one in the previous  paragraph  shows only   that 
\[
\deg (A_n) \le 1 \quad\text{and} \quad \deg(B_n) \le 1.
\]
And indeed, given a generically irreducible quasi-admissible module with equal degrees, we can obtain a new one by simply exchanging any number of the  $A_n$ with  $B_n$, so that  the new action of $\boldsymbol{\mathfrak {g}}$ is defined by
\[
Xf_n    =B_n  f_{n+2}\quad \text{and} \quad 
Yf_{n+2}  = z^{-1}A_{n} f_{n}
\]
for these chosen values of $n$.  The new module will be isomorphic to the old one if and only if $A_n$ is a multiple of $B_n$ for all these chosen $n$, which is impossible for more than two choices of $n$.

So by starting with one  generically irreducible module with equal degrees, we construct in this way an uncountable family of modules all with the same Casimir, weight and degree invariants.

\subsection{Some classes of generically irreducible modules}

The observations made in the previous two paragraphs show that there are too many generically irreducible modules to admit a reasonable classification (although  it is indeed possible to carry out that classification).  For this reason it seems worthwhile to consider some subclasses, and this is what we shall do in this section.   Similar subclasses were studied before in \cite{Noort}.

All of the subclasses will consist of modules for which $\deg(\mathcal F_n)$ always differs from $\deg(\mathcal F_{n+2})$.  The first (which is actually an infinite family, parametrized by weights $k$) is as follows: 
\begin{enumerate}[\rm I{${}_k$}.]

\item  Modules with   $\mathcal F_k$ nonzero and 
	\begin{enumerate}[\rm (i)]
		\item   $\deg(\mathcal F_{n+2}) = \deg (\mathcal F_n) -1$ whenever $n\ge k$,
		\item $\deg(\mathcal F_{n-2}) = \deg (\mathcal F_n) -1$ whenever $n\le k$.
	\end{enumerate}	
\end{enumerate}
To say that $\deg(\mathcal F_{n\pm 2}) = \deg (\mathcal{F}_n ) - 1$ is to say that the morphism
\[
\boldsymbol{\mathfrak {g}}_{\pm 2} \otimes _{O_X} \mathcal F_n \longrightarrow \mathcal F_{n\pm 2}
\]
is an isomorphism.  As a result, the modules in class $I_k$ can alternately be described as those modules that are generated, as algebraic families of $(\boldsymbol{\mathfrak{g}}, \boldsymbol{K})$-modules or equivalently as ordinary  Harish-Chandra modules in every fiber, by their $k$-th isotypical parts. 

Similarly we can consider the ``dual'' class
\begin{enumerate}[\rm I{${}_k$}.]
	  \setcounter{enumi}{1}
\item  Modules with   $\mathcal F_k$ nonzero and 
	\begin{enumerate}[\rm (i)]
		\item   $\deg(\mathcal F_{n+2}) = \deg (\mathcal F_n) +1$ whenever $n\ge k$,
		\item $\deg(\mathcal F_{n-2}) = \deg (\mathcal F_n) + 1$ whenever $n\le k$.
	\end{enumerate}
 \end{enumerate}
 These are the modules $\mathcal F$ for which every $(\boldsymbol{\mathfrak{g}}, \boldsymbol{K})$-submodule contains the $k$-th isotypical part of $\mathcal F$.
 
The classes $I_k$ and $II_k$ exhibit phenomena seen in individual (reducible) principal series representations, which is why they might deserve further study. On the other hand the following two classes seem very natural in their own rights: 
\begin{enumerate}[\rm I.]
	  \setcounter{enumi}{2}
	  
\item   
Modules with $\deg(\mathcal F_{n+2}) = \deg (\mathcal F_n)  + 1$  for all $n$.

\item 
Modules  with  $\deg(\mathcal F_{n+2}) = \deg (\mathcal F_n) -1$ for all $n$.

\end{enumerate}
 One might say that these are the modules generated by high enough, or low enough, isotypical parts, respectively.

 \begin{remark}
 The Picard group  of isomorphism classes of invertible sheaves  on $X$  acts by tensor product on   algebraic families of Harish-Chandra modules over $(\boldsymbol{\mathfrak{g}}, \boldsymbol{K})$.  For this reason it would be harmless to assume in addition    that for example the $k$-th weight space has degree zero in cases $I_k$ and $II_k$, or that the degree of $\mathcal {F}_n$ is precisely $\lfloor \frac n2\rfloor$ or $-\lfloor \frac n2\rfloor$ in cases $III$ or $IV$, respectively.
\end{remark}

\subsubsection{The principal series case}
\label{sec-principal-series-case}
In this paragraph we shall show how to construct   a generically irreducible $(\boldsymbol{\mathfrak{g}},\boldsymbol{K})$-mod\-ule $\mathcal F$ with weights  $2\mathbb{Z}$ or $2\mathbb{Z}{+}1$ that lies in any of the four classes above and has  any value of the   Casimir section  
\[
c_1z+c_0+ c_{-1}z^{-1} ,
\] 
except for certain constant integral values, which can never occur.

The excluded values are the constants $m(m+2)$, where  $m\in \mathbb Z$ has the same parity as the weights of $\mathcal F$.  To see this, we refer  to the notation from Section~\ref{sec4.8}. When the Casimir section assumes the constant value $m(m+2)$, the \emph{Casimir equation}  relating the Casimir section to the polynomials $A_n$ and $B_n$ is 
\[
m(m+2) = n(n+2) + A_n(z)B_n(z)z^{-1}.
\]
So one of $A_m$ or $B_m$ is zero, which cannot happen in a generically irreducible module.

If the excluded values are avoided, then a module in each of the classes $I$-$IV$ can be constructed using the methods of Section~\ref{sec4.8}, as follows.
The   Casimir equation never forces the product $A_nB_n$ to be the zero polynomial, and we can always choose a solution with one of $A_n$ or $B_n$ to be $1$, and the other to be a quadratic polynomial,  in such a way that the degrees of $A_n$ and $B_n$ satisfy   the inequalities in Section~\ref{sec4.8}.

When the inequalities are satisfied, the polynomials $A_n$ and $B_n$   define   morphisms of sheaves 
\[
 \boldsymbol{\mathfrak{g}}_{2}\otimes \mathcal {F}_{n} \longrightarrow  \mathcal {F}_{n+2}
\]
and
\[
 \boldsymbol{\mathfrak{g}}_{-2}\otimes \mathcal {F}_{n+2} \longrightarrow  \mathcal {F}_{n} .
\]
The fact that $A_n$ and $B_n$ solve the Casimir equation implies that these morphisms define a generically irreducible, quasi-admissible family of $(\boldsymbol{\mathfrak{g}}, \boldsymbol{K})$-modules.

 \subsubsection{The non-principal series cases}
 \label{sec-non-principal-case}
 
 When the weights of a quasi-admissible and   generically irreducible family of $(\boldsymbol{\mathfrak{g}}, \boldsymbol{K})$-modules fall into one of the non-principal series cases listed in Table~\ref{tab1new}, the situation is completely different from the one described above. 
 
\begin{lemma}
If the weights of  a generically irreducible $(\boldsymbol{\mathfrak{g}},\boldsymbol{K})$-module, $\mathcal{F}$, coincide with those of a discrete series, limit discrete series, or irreducible finite dimensional $(sl_2(\mathbb{C}),H)$-module, then the  Casimir section $C$ acts by multiplication by a constant function. The constant is 
\begin{eqnarray}
&& \begin{cases}
l^2-2|l| & \text{for discrete or limit discrete series, and}\\
k^2+2k & \text{for finite-dimesional representations,}
\end{cases}
\end{eqnarray}
where we refer to \textup{Table~\ref{tab1new}} for the indexing.
\label{lem15}
\end{lemma}

\begin{remark}
Note that $\ell^2 - 2 |\ell | = k(k+2)$ when $k>0$ and $ | \ell|  = k+2 $.
\end{remark}

\begin{proof}
In all of these cases there is an extreme (either highest or lowest) weight space, and it is a simple matter to directly compute the action of the Casimir for $SL(2)$ on this weight space, with the results indicated. \end{proof}

As for the existence of modules as above, it is a simple matter to construct a generically irreducible $(\boldsymbol{\mathfrak{g}},\boldsymbol{K})$-module with the indicated Casimir section in each of the classes $I$-$IV$ above (and indeed in any case where the degrees of successive weight modules differ by no more than $1$).

 \subsubsection{Classification of generically irreducible    modules}

 Let us summarize.    
 
 \begin{theorem}
 In each of the classes $I$-$IV$, the classification of generically irreducible quasi-admissible $(\boldsymbol{\mathfrak{g}},\boldsymbol{K})$-modules up to isomorphism and an action of the Picard group of $X$ is as follows:
 \begin{enumerate}[\rm (i)]
 \item The set of weights of any  generically irreducible quasi-admissible module is one of the possibilities listed in Table~\textup{\ref{tab1new}}.
 \item When there is an extreme weight, the Casimir section is determined by the extreme weight, as in Paragraph~\textup{\ref{sec-non-principal-case}}, and there is a unique module, up to isomorphism, with the given weights in the given class.
 \item When there is no extreme weight, the set of weights is either $2\mathbb Z$ or $2\mathbb Z{+}1$. The Casimir section cannot be one of the excluded possibilities listed in Paragraph~ \textup{\ref{sec-principal-series-case}}, but there is a unique module, up to isomorphism,  in each of the classes,
 with the given weights  and with  any other value of the Casimir section.\qed
  \end{enumerate}

 \end{theorem}


\bibliography{references}
\bibliographystyle{alpha}

\end{document}